\documentclass{birkjour}
\numberwithin{equation}{section}

\newtheorem{thm}{Theorem}[section]
\newtheorem{theorem}[thm]{Theorem}
\newtheorem{lemma}[thm]{Lemma}
\newtheorem{corollary}[thm]{Corollary}
\newtheorem{proposition}[thm]{Proposition}

\newtheorem{remark}[thm]{Remark}

\newtheorem{definition}[thm]{Definition}

\newtheorem{example}[thm]{Example}

\newtheorem{defn-thm}[thm]{Definition-Theorem}

\renewcommand{\L}{{\mathbb L}}

\newcommand{\btheorem}{\begin{theorem}}
\newcommand{\etheorem}{\end{theorem}}
\newcommand{\bproposition}{\begin{proposition}}
\newcommand{\eproposition}{\end{proposition}}
\newcommand{\bdefinition}{\begin{definition}}
\newcommand{\edefinition}{\end{definition}}
\newcommand{\bcorollary}{\begin{corollary}}
\newcommand{\ecorollary}{\end{corollary}}
\newcommand{\bproof}{\begin{proof}}
\newcommand{\eproof}{\end{proof}}
\newcommand{\bremark}{\begin{remark}}
\newcommand{\eremark}{\end{remark}}
\newcommand{\eexample}{\end{example}}
\newcommand{\bexample}{\begin{example}}

\newcommand{\elemma}{\end{lemma}}
\newcommand{\blemma}{\begin{lemma}}

\renewcommand{\phi}{\varphi}

\newcommand{\ee}{\end{eqnarray*}}
\newcommand{\be}{\begin{eqnarray*}}

\newcommand{\beq}{\begin{equation}}
\newcommand{\eeq}{\end{equation}}

\newcommand{\bd}{\begin{enumerate}}
\newcommand{\ed}{\end{enumerate}}

\renewcommand{\sf}{\textsf}

\begin{document}

%
%
%
%
%
%
%
%
%

\title[Non-maximal volume growth]
 {Complete open K$\ddot{a}$hler manifolds with nonnegative bisectional curvature and non-maximal volume growth}

\author{James W. Ogaja}

\address{
University of Califonia Riverside\\
900 University Avenue\\
Riverside, CA 92521, USA}

\email{jogaj001@ucr.edu}



\date{January 1, 2004}

\begin{abstract}
It is still an open problem that a complete open K$\ddot{a}$hler manifold with positive bisectional curvature is Stein. This paper partially resolve the problem by putting a restriction to volume growth condition. The partial solution here improves the observation in ([8], page 341). The improvement is based on assuming a weaker volume growth condition that is is not sufficiently maximal.  

\end{abstract}

\maketitle
\section{Introduction}
One of the most useful tool in studying structures of nonnegatively curved complete open manifolds is the Busemann function. The spherical  Busemann function is defined as  $$\displaystyle{b_p(x) = \lim_{r\to\infty}\{r - d(x, \partial(B(p,r))\}},$$

In the process of proving soul theorem in [3], Cheeger and Gromoll proved that a complete open Riemaniann manifold with nonnegative sectional curvature admits a convex and exhaustion Busemann function, $b_p$.

It still remains unknown whether complete open manifolds with nonnegative Ricci curvature admit exhaustion Busemann function, except with the restriction to maximum volume growth as was proved by Shen in [7]. In this paper we drop the maximum volume growth condition and adapt a weaker condition.

Let $S_pM\subset T_pM$ be a unit tangent sphere in the tangent space $T_pM$ for a point $p\in M$. For any subset $N\subset S_pM$ , define $$C(N)=\{ q\in M \ | \ \text{there is a minimizing geodesic $\gamma$ from p to q such that} \ \gamma'(0)\in N\}$$ to be the cone over $N$. The restriction of a geodesic ball of radius $r$ centered at $p$ to C(N) is denoted by $$B_N(p,r) = B(p,r)\cap C(N)$$.

 Let $\Sigma = \{ v\subset S_pM \ | \ \text{exp}_p(rv) : [0,\infty)\to M  \text{ is a ray}\}$. A cone of rays is defined by $C(\Sigma)$. Consequently, $$B_{\Sigma}(p,r) = B(p,r) \cap C(\Sigma).$$

\vspace{0.1in}

From lemma 4 in [5] we have

\begin{lemma}[{\cite[Ordway-Stephens-Yang]{OSY}}]
 Let $M^n$ be a complete open manifold with $\text{Ric}_{M} \geq 0$. Suppose that $M$ has a maximum volume growth i.e $$\lim_{r\to\infty} r^{-n}Vol(B(p,r)) = \alpha_M, \ \alpha_M > 0$$ then $$\lim_{r\to\infty}r^{-n}Vol(B_\Sigma(p,r)) = \alpha_M$$
\end{lemma}

\vspace{0.1in}

By limit properties, we obtain the following corollary

\begin{corollary}
 Let $M^n$ be a complete open manifold with $\text{Ric}_{M} \geq 0$. Suppose that $M$ has a maximum volume growth. Then $$Vol(B_{\Sigma}(p,r))\sim Vol(B(p,r)),$$
\end{corollary}

 $$\ \sim \ means \ asymptotic$$

\vspace{0.1in}

It is essential to note that nonnegative Ricci curvature ensures that the volume growth condition in corollary 1.2 above is independent of the base point: let $p,q\in M^n$ and $d = d(p,q)$. Then it is clear that $B(p, r)\subset B(q, r + d)$ and $B(q, r)\subset B(p, r + d)$.
By Bishop-Gromov volume comparison theorem, 

\begin{eqnarray*}\lim_{r\to\infty}\text{inf}\frac{Vol(B_{\Sigma}(p,r))}{Vol(B(p_1,r))}&\geq &\lim_{r\to\infty}\text{inf}\left\{\left[\frac{r}{r + d}\right]^n\frac{Vol(B_{\Sigma}(p,r + d))}{Vol(B(p_1,r))}\right\}\\&\geq & \lim_{r\to\infty}\text{inf}\left\{\left[\frac{r}{r + d}\right]^n\frac{Vol(B_{\Sigma}(q,r)}{Vol(B(p_1,r))}\right\}\\&\geq & \lim_{r\to\infty}\left[\frac{r}{r + d}\right]^n\lim_{r\to\infty}\text{inf}\frac{Vol(B_{\Sigma}(q,r)}{Vol(B(p_1,r))}\\&= & \lim_{r\to\infty}\text{inf}\frac{Vol(B_{\Sigma}(q,r)}{Vol(B(p_1,r))}\end{eqnarray*}

Likewise  \begin{eqnarray*}\lim_{r\to\infty}\text{inf}\frac{Vol(B_{\Sigma}(p,r))}{Vol(B(p_1,r))}\leq\lim_{r\to\infty}\text{inf}\frac{Vol(B_{\Sigma}(q,r)}{Vol(B(p_1,r))}\end{eqnarray*}

\begin{lemma}
Let $M^n$ be a complete open manifold with $Ric_M\geq 0$. For a fixed $p_1\in M$, the volume growth $$\lim_{r\to\infty}\text{inf} \ [Vol(B_{\Sigma}(p,r))][Vol(B(p_1,r))]^{-1} = \alpha(n)$$ is independent of the base point $p\in M$.
\end{lemma}

\vspace{0.1in}

 The converse to corollary 1.2 above is not true. In other words, the volume growth condition
\begin{eqnarray}Vol(B_{\Sigma}(p,r))\sim Vol(B(p,r)),\end{eqnarray} does not necessarily imply maximum volume growth.

 For example, the vertex 0 of a paraboloid $M\subset\mathbb{R}^{n + 1}$ has an empty cut locus. Thus volume growth condition (1.1) holds at 0 and extends to other points by lemma 1.3 above. On the other hand, as a special case of lemma 4.1 in [6], the paraboloid $M$  in $\mathbb{R}^{n + 1}$ defined by $M = \{ (x_1, x_2, \cdot \cdot \cdot, x_n, z) : z = x_1^2 + x_2^2 + \cdot \cdot \cdot + x_n^2 \}$  has a volume growth of at most $r^{\frac{n + 1}{2}}$ which is not maximal. Furthermore, we can creat a non-empty cut locus of the point $0$ at the same time maintaining positive curvature and manifesting volume growth conditions like that of (1.1).

\begin{example}
Consider $\mathcal{M} = \{ (x_1, x_2, \cdot \cdot \cdot, x_n, z) : z = x_1^2 + x_2^2 + \cdot \cdot \cdot + x_n^2 \}\subset\mathbb{R}^{n + 1}$. $(\mathcal{M}, ds_\mathcal{M}^2)$ is a complete open manifold with positive Ricci curvature ($n>2$). Here, $ds_M^2$ is an induced Euclidean metric. For $0\neq q\in\mathcal{M}$, let $D_l(q)$ be a geodesic ball of radius $l$ centered at $q$. Consider a smooth function $f: \overline{D_l(q)}\to\mathbb{R}$. For a small neiborhood $U$ of $\overline{D_l(q)}$, there exists a smooth function $h: \mathcal{M}\to\mathbb{R}$ such that $h|_{\overline{D_l(q)}} = f$ and $\text{supp}\ h\subset U$. For $\varepsilon>0$, denote $\mathcal{M_{\varepsilon}} = (\mathcal{M}, ds_\mathcal{M}^2 + \varepsilon hds_{\mathcal{M}}^2)$. We can choose $\varepsilon$ small enough such that the Ricci curvature remains positive throughout $\mathcal{M}$ and an extension $\gamma: [0,\infty) \to \mathcal{M_{\varepsilon}}$ of a minimizing geodesic from $0$ to $q$ leaves $D_l(q)$ and intersect a ray at a point. It follows that the cut locus of the point $0$ is no longer empty. Since only rays intersecting and neighboring $D_l(q)$ are affected in this new manifold, for a fixed $a\approx 1$, $a\leq1$, we can choose $l>0$ and $\varepsilon>0$ small enough such that $$\lim_{r\to\infty}[Vol(B_{\Sigma}(p,r))][Vol(B(p,r))]^{-1} = a$$.
\end{example}

Given two real valued functions $f,g:\mathbb{R}\to \mathbb{R}$. Denote the limit $$\lim_{r\to\infty}f(r)[g(r)]^{-1} = a, \ a>0 $$ 
if it exists by $f\sim_ag$.

Now we extend the result by Shen in [7] by replacing the maximum volume growth condition with a weaker volume growth condition.

\begin{lemma}\label{lemm:cute}
 Let $M$ be a complete open manifold with $\text{Ric}_{M} \geq 0$. Let $[9^n - 1]9^{-n}<a\leq 1$ where $n = dim_{\mathbb{R}}M$. If
$$\lim_{r\to\infty}\text{inf} \ [Vol(B_{\Sigma}(p,r))][Vol(B(p,r))]^{-1} = a,$$ then for any $t\in \mathbb{R}$, $b_p^{-1}(t)$ is compact . 
\end{lemma}

The following theorem is the main result in this paper:

\begin{theorem} \label{thm:cute}
Let $M$ be a complete open K$\ddot{a}$hler manifold with nonnegative bisectional curvature. Then $M$ is a Stein manifold if the followings holds
\begin{itemize}
\item[(a)]The bisectional curvature is positive outside a compact set
\item[(b)] $$Vol(B_{\Sigma}(p,r))\sim_a Vol(B(p,r)),$$ 

where $[9^{2n} - 1]9^{-2n}<a\leq1$ and $n = dim_{\mathbb{C}}M$
\end{itemize}
 \end{theorem}

\section{Proofs}

\bigskip

We will prove Lemma 1.5 first then Theorem 1.6.

 \begin{proof}[Proof of Lemma \ref{lemm:cute}]

 Proving  by contradiction, we assume that $b^{-1}(t)$ is non-compact and then show that the assumed volume growth condition doesn't hold.

Define the excess function for two points $p,q$ as $$e_{p,q} = d(p,x) + d(x,q) - d(p,q).$$

 By the triangle inequality, we have that 
 \begin{eqnarray}e_{p,q}(x)\leq2h(x)\end{eqnarray}

Denote $r_p(x) = d(p,x)$. Assume that the minimizing geodesic between $p$ and $q$ is part of a ray emanating from $p$. Now, taking the limit of inequality (2.1) as $q$ goes to infinity, we end up with the following inequality 

\begin{eqnarray}r_p(x) - \lim_{t\to\infty}\{t - d(x,\gamma(t))\}\leq 2h_{\gamma}(x),\end{eqnarray} where $h_{\gamma}(x)$ is a distance from $x$ to a ray $\gamma$ emenating from $p$.  Since $$r_p(x) - b_p(x)\leq r_p(x) - \lim_{t\to\infty}\{t - d(x,\gamma(t))\},$$ for each ray $\gamma$ emanating from $p$, inequality (2.2) implies that \begin{eqnarray}r_p(x) - b_p(x)\leq 2h_{\gamma}(x)\end{eqnarray} 

Let $h_p(x) = d(x,Rp)$, where $R_p$ is a union of rays emanating from $p$. Since inequality (2.3) holds for any ray $\gamma$, we have that

\begin{eqnarray}r_p(x) - b_p(x)\leq2 h_p(x)\end{eqnarray}

Next, note that $$C(\Sigma)\cap C(\Sigma^c) = \emptyset.$$

Therefore, for any $r>0$ and $p\in M$, we have that

$$B_{\Sigma}(p,r)\cap B_{\Sigma^c}(p,r) = \emptyset$$

 Observe that $B(x, h_p(x))\subset C(\Sigma^c)$. It follows that $B(x, h_p(x))\subset B_{\Sigma^c}(p, r_p(x) +  h_p(x))$. 

\vspace{0.1in}

Since $b_p$ is exhaustion whenever $h_p$ is bounded, we assume that $h_p$ is unbounded. Due to noncompactness of $b_p^{-1}(t)$, we can construct a diverging  sequence $\{x_m\}\subset b_p^{-1}(t)$. Consequently, $\{h_p(x_m)\}$ is a divergence sequence. 

\vspace{0.1in}

 Denote $h_m = h_p(x_m)$ and $r_m = r(x_m)$. By Bishop-Gromov volume comparison theorem, $$\displaystyle{\frac{Vol(B_{\Sigma^c}(p, r_m - h_m))}{Vol(B_{\Sigma^c}(p, r_m + h_m))} \geq \left[\frac{r_m - h_m}{r_m + h_m}\right]^n }$$

It is easy to verify that \begin{eqnarray}B(x_m, h_m)\subset B_{\Sigma^c}(p,r_m + h_m)\backslash B_{\Sigma^c}(p,r_m - h_m)\end{eqnarray}

and that

\begin{eqnarray} Vol(B(x_m, h_m) &\leq& Vol(B_{\Sigma^c}(p,r_m + h_m)) - Vol( B_{\Sigma^c}(p,r_m - h_m)) \nonumber \\ &\leq& \left\{1 - \left[\frac{r_m - h_m}{r_m + h_m}\right]^n\right\}Vol(B_{\Sigma^c}(p,r_m + h_m)) \nonumber\\ & \leq & Vol(B_{\Sigma^c}(p, 3h_m + a))  \end{eqnarray}

\vspace{0.1in} 

 Inequality 2.6 is due to the fact that $h\leq r$ and $$r_p(x) - b_p(x)\leq2 h_p(x)$$

 In particular \begin{eqnarray*}r_p(x) + h(x) \leq 3h(x) + a , \ \text{when $x\in b_p^{-1}(a)$}\end{eqnarray*}

 Now, denote $r_1(x) = d(x_1,x)$. By triangle inequality and (2.4),

\begin{eqnarray}\lim_{m\to\infty}\text{sup} \ \frac{r_1(x_m)}{h_p(x_m)} & \leq &\lim_{m\to\infty}\text{sup} \ \frac{r_1(p)}{h_p(x_m)} + \lim_{l\to\infty}\text{sup}\ \frac{r_p(x_m)}{h_p(x_m)}\nonumber \\ & \leq & 2\end{eqnarray}

 Also note that 

\begin{eqnarray}B(x_1, h_m)\subset B(x_m, h_m + r_1(x_m)) \end{eqnarray}

 By volume comparison theorem we obtain

\begin{eqnarray} Vol(B(x_m,h_m))\geq \left[\frac{h_m}{h_m + r_1(x_m)}\right]^nVol(B(x_m,h_m + r_1(x_m))\end{eqnarray}

 Denote $f_p(r) = Vol(B(p,r)) \text{for a fixed $p\in M$}$. 
 From (2.7), (2.8), and (2.9), we have

\begin{eqnarray}&\underset{m \to \infty}{\lim}&\text{inf}\frac{Vol(B(x_m, h_m))}{f_p(h_m)} \nonumber\\ &\geq &  \lim_{m \to \infty}\text{inf}\left[\left(\frac{h_m}{(h_m + r_1(x_m))}\right)^n\frac{Vol(B(x_m, h_m + r_1(x_m)))}{f_p(h_m)} \right] \nonumber \\ &\geq &  \lim_{m \to \infty}\text{inf}\left[\left(\frac{h_m}{(h_m + r_1(x_m))}\right)^n\frac{Vol(B(x_1, h_m))}{f_p(h_m)} \right]\nonumber \\ & \geq & \displaystyle{\underset{m\to\infty}{\lim}\text{inf}\frac{1}{\left(1 + \frac{r_1(x_m)}{h_m}\right)^n}}\lim_{m \to \infty}\text{inf}\frac{Vol(B(x_1, h_m))}{f_p(h_m)}\nonumber \\ &\geq & 3^{-n}\end{eqnarray}

The last inequality is due to the fact that the volume growth $$\lim_{m \to \infty}\text{inf}\ [Vol(B(x_1, h_m))][f_p(h_m)]^{-1}$$

 is independent of the base point $x_1$.

From inequalities (2.6), (2.10), and the volume comparison theorem, we have 

\begin{eqnarray} 3^{-n} &\leq& \lim_{m \to \infty}\text{inf}\frac{Vol(B(x_m, h_m))}{f_p(h_m)}\nonumber \\ &\leq &  \lim_{m\to\infty}\text{inf}\frac{Vol(B_{\Sigma^c}(p, 3h_m + a))}{f_p(h_m)} \nonumber \\ &\leq & 3^n \lim_{m\to\infty}\text{inf}\frac{Vol(B_{\Sigma^c}(p, h_m))}{f_p(h_m)} \end{eqnarray}

 Which leads to the inequality \begin{eqnarray}\lim_{m\to\infty}\text{inf} \ [Vol(B_{\Sigma^c}(p, h_m))][f_p(h_m)]^{-1}\geq 9^{-n}\end{eqnarray}

 Since $$Vol(B(p,r) = Vol(B_{\Sigma}(p,r)) + Vol(B_{\Sigma^c}(p,r)),$$ the volume growth condition assumption implies that \begin{eqnarray} \lim_{r\to\infty}\text{inf}\ [Vol(B_{\Sigma^c}(p, r))][f_p(r)]^{-1} < 9^{-n}\end{eqnarray}

 Evidently, inequality (2.12) contradicts inequality (2.13). Hence $b_p^{-1}(t)$ must be compact.

 \end{proof}

 \begin{proof}[Proof of Theorem \ref{thm:cute}] The $Ricci$ curvature is nonnegative and positive outside a compact set because the bisectional curvature is assumed. The Busemann function $b_p$ is a continuous plurisubharmonic exhaustion by lemma 1.5 and a result by H.Wu in [9]. In the same paper (Theorem C [9]), it follows that there exist a strictly plurisubharmonic exhaustion function. This completes the proof.

\end{proof}

\section{Applications}

Let $H_k(M,\mathbb{Z})$ denote the k-th singular homology group of $M$ with integer coefficients. It is well known that if $M$ is a complete proper Riemannian n-dimensional manifold with $Ric_M\geq 0$, then using Morse theorem, $M$ has the homotopy type of a CW complex with cells each of dimension $\leq n-2$ and $H_i(M,\mathbb{Z}) =0$, $i\geq n -1$. ([6], [4])

As an application of lemma 1.5, we have the following result.

\begin{corollary}
Let $(M,g)$ be a complete open manifold with $Ric_M\geq 0$ .  If  $$Vol(B_{\Sigma}(p,r))\sim_a Vol(B(p,r)),$$ 
where $[9^n - 1]9^{-n}<a\leq1$, then $M$ has the homotopy type of a $CW$ complex with cells each of dimension $\leq n-2$. In particular, $H_i(M,\mathbb{Z}) =0$, $i\geq n -1$
\end{corollary}

It is also known that if $M$ is a Stein manifold of n-complex dimension, then the homology groups $H_k(M,\mathbb{Z})$ are zero if $k>n$ and $H_n(M,\mathbb{Z})$ is torsion free  (theorem 1 [1]), [2]. As an application of theorem 1.6, we have the following result.

\begin{corollary}
Let $M$ be a complete open K$\ddot{a}$hler manifold with nonnegative bisectional curvature. If the followings holds
\begin{itemize}
\item[(a)]The bisectional curvature is positive outside a compact set
\item[(b)] $$Vol(B_{\Sigma}(p,r))\sim_a Vol(B(p,r)),$$ 

where $[9^{2n} - 1]9^{-2n}<a\leq1$ and $n = dim_{\mathbb{C}}M$
\end{itemize} then $$H_k(M,\mathbb{Z}) = 0, \ \text{for} \ k>n$$ and $H_{n}(M,\mathbb{Z})$ is torsion free
\end{corollary}


%
%
%

\vspace{0.2in}

\noindent Acknowledgement: Thanks to Professor Bun Wong for his advice.

\end{document}